\documentclass[12pt, twoside, leqno]{article}

\usepackage{amsmath,amsthm}
\usepackage{amssymb}

\usepackage{enumerate}

\usepackage{graphicx}

\newtheorem{thm}{Theorem}[section]
\newtheorem{cor}[thm]{Corollary}
\newtheorem{lem}[thm]{Lemma}
\newtheorem{prop}[thm]{Proposition}

\theoremstyle{definition}

\numberwithin{equation}{section}

\newcommand{\R}{\mathbb{R}}
\newcommand{\N}{\mathbb{N}}

\begin{document}


\baselineskip=17pt


\title{Density of the set of symbolic dynamics with all ergodic measures supported on periodic orbits}

\author{Tatiane Cardoso Batista\\
Universidade Tecnol\'ogica Federal do Paran\'a\\ 
85053-525 Guarapuava, Brazil\\
E-mail: tatianebatista@utfpr.edu.br
\and 
Juliano dos Santos Gonschorowski\\
Universidade Tecnol\'ogica Federal do Paran\'a\\ 
85053-525 Guarapuava, Brazil\\
E-mail: julianod@utfpr.edu.br
\and 
Fabio Armando Tal\\
Instituto de Matem\'atica e Estat\'istica\\ 
Universidade de S\~ao Paulo\\
05508-090 S\~ao Paulo, Brazil\\
E-mail: fabiotal@ime.usp.br}

\date{January 15, 2015}

\maketitle


\renewcommand{\thefootnote}{}

\footnote{2010 \emph{Mathematics Subject Classification}: Primary XXXX; Secondary YYYY.}


\renewcommand{\thefootnote}{\arabic{footnote}}
\setcounter{footnote}{0}


\begin{abstract}
Let $K$ be the Cantor set. We prove that arbitrarily close to a homeomorphism $T:K\rightarrow K$ there exists a homeomorphism $\widetilde T:K\rightarrow K$ such that the $\omega$-limit of every orbit is a periodic orbit.
We also prove that arbitrarily close to an endomorphism $T:K\rightarrow K$  there exists an endomorphism $\widetilde T:K\rightarrow K$ close to $T$ such that every orbit is finally periodic.
\end{abstract}

\section{Introduction}
Given $X$ a topologic space, a dynamic $T:X\rightarrow X$ and an observable $f:X\rightarrow\R$, both continuous functions, the fundamental question in ergodic optimization is to determine, among the set $M_T(X)$ of all  $T$-invariant Borel probability measures, which measures maximize the functional $F_f:M_T(X)\rightarrow\R$ definided by:
$$\displaystyle F_f(\mu)=\int_Xf\ d\mu.$$
A measure maximizing this functional is usually called a $f$ maximizing measure.

This relatively new field of study has seen a fast development in the last decade, and several interesting lines of research have been pursued, see for instance \cite{art3} and references therein. Among the different lines of research in this subject, one that has received great attention is the search to determine what is the typical support of maximizing measures, where typical is of course context dependent. For instance, one of the most relevant conjectures, which was later proved by Contreras in \cite{gonzalo}, considered the case where there existed a fixed expansive $T$ and asked if, in the space of Lipschitz functions from $X$ to $\R$, the set observables for which there exists a single $f$-maximizing measure, with this measure supported on a periodic orbit, is a $G_{\delta}$.

In the same spirit but in a different context, some works search to determine what is the expected behaviour of maximizing measures when the observable $f$ is fixed, but the dynamics varies in a given space. This has been done for instance when $X$ is a compact Riemanian manifold, where in \cite{arthom,artsf} the typical behaviour when $T$ belongs to the set of homeomorphisms of $X$ is studied, and in \cite{artend, juju} this is done with $T$ varying in the space of all continuous surjections of $X$.
This work was born out of a similar question, to study this behaviour when $X$ is the cantor set $K$, but the investigation led to a result which is somewhat more general, dealing only with the dynamics in $K$ and bypassing discussions on the observable functions.

 The Cantor set $K$ is defined as a totally disconnected, perfect and compact metric space, and a classical result states that any two such sets are homeomorphic. Since all of our results are topological we can consider any realization of the Cantor set, and we will usually work with the set $\Sigma_N$ of  one sided or two sided sequences of $N$-symbols with the usual metric. We consider the sets $End(K)$ of endomorphisms of $K$, that is, all continuous surjections of $K$, with the metric $D\left(T,\widetilde T\right)=\displaystyle\max_{x\in K} d\left(T(x),\widetilde T(x)\right)$ and also the subset $Hom(K)$ of all homeomorphisms of $K$, endowed with the induced metric. Our main results are: 

\begin{thm}\label{res1}
Given $T\in End(K)$ and $\varepsilon>0$, there exists $\widetilde T\in End(K)$  such that $D\left(T,\widetilde T\right)<\varepsilon$ and every orbit of $ \widetilde T$ is finally
periodic\footnote{\normalsize{We say that the orbit of a point $ x \in K $ is
finally periodic by $T$ if there exist $j,N>0$ such that
$T^{N+j}(x)=T^j(x)$}}.
\end{thm}

\begin{thm}\label{res2}
Given $T\in Hom(K)$ and $\varepsilon>0$, there exists $\widetilde T\in Hom(K)$ such that $D\left(T,\widetilde T\right)<\varepsilon$ and the $\omega$-limit of every orbit of $\widetilde T$ is a periodic orbit.
\end{thm}

As direct consequence of these theorems we are able to obtain a positive answer for the ergodic optimization problem:

\begin{cor}\label{co:intro}
Let $K$ be a Cantor set. Given an endomorphism (respectively homeomorphism) $T:K\rightarrow K$, a continuous function $f:K\rightarrow \R$ and $\varepsilon>0$, there exists an endomorphism (respectively homeomorphism) $\widetilde T:K\rightarrow K$ with
$$D\left(T,\widetilde T\right)=\displaystyle\max_{x\in K} d\left(T(x),\widetilde T(x)\right)<\varepsilon$$
\noindent and such that $\widetilde T$ has a $f$-maximizing measure supported on a periodic orbit.
\end{cor}

This short note is designed to be self-contained. There study of typical dynamics in the Cantor set apart from any ergodic optimization result has already some important literature. For instance, the works \cite{nov3, nov4, nov} provide a much stronger result, showing the existence of a conjugacy class in $Hom(K)$ which itself is generic, and describing this conjugacy class. Parts of Theorems \ref{res1} and \ref{res2} could be derived more directly from their work, but their proofs are more involved.

\section{Preliminaries}

We begin by listing some simple and immediate properties of the Cantor set:

\begin{prop}\label{div}
A Cantor set can be partitioned into $N$ disjoint nonempty Cantor sets.
\end{prop}

\begin{prop}\label{eps}
Given $\varepsilon>0$, there exists $M>0$ and disjoint subsets $K_j\subset K$ such that
\begin{enumerate}[\upshape (i)]
\item $K_j$ is a Cantor set, with $0\leq j\leq M$;
\item $K\displaystyle=\bigcup_{j=1}^{M}K_j$;
\item For each $j=1,2,3,...,M$, diam$\left(K_j\right)<\varepsilon$.
\end{enumerate}
\end{prop} 

\begin{prop}\label{seq}
There exists a disjoint sequence $(K_m)_{m\in\N}, \, K_m\subset K$ and a point $p\in K\setminus \bigcup_{m\in\N}K_m$ such that
\begin{enumerate}[\upshape (i)]
\item $K_m$ is a Cantor set, for any $m\in\N$;
\item $K\displaystyle=\left(\bigcup_{m=1}^{\infty}K_m\right)\cup\{p\}$;
\item Given $\varepsilon>0$, there exists $m_0\in\N$ such that, for $m>m_0$, diam$\left(K_m\right)<\varepsilon$.
\end{enumerate}
\end{prop}

\section{Endomorphisms}

Let $\Sigma_N^{+}$ be the space of the sequences $(x_n)_{n\in\N}$ with $x_n\in \{1,2,...,N\}$, and consider the usual metric:
$$d(x,y)=2^{-\min \{i:x_i\neq y_i\}} \ \textnormal{for} \ x,y\in\Sigma_N^{+}.$$

Define $$W_i=\{x: \ x=(i,x_2,x_3,...)\in\Sigma_N^{+}\}.$$

\begin{lem}
Let $T\in End(\Sigma_N^{+})$ be such that for all $1\leq i \leq N$, exists $x^i\in W_i$ with $T(x^i)\in W_i$.
Then there exists $\widetilde T\in End(\Sigma_N^{+})$ such that:
$$\widetilde T(x)\in W_i \Longleftrightarrow T(x)\in W_i,$$
\noindent and every orbit of $\widetilde T$ is finally periodic. 
\end{lem}

\begin{proof}
By the continuity of $T$, for all $i=1,...,N$, there exists $l_i$ such that, if  $R_i=x^i_2, x^i_3,...,x^i_{l_i}$, then for all $x$ of the form$(i,R_i,x_{l_i+1},x_{l_i+2},...),\,   T(x)\in W_i.$

For each $i=1,...,N$, define $W_{i,i}=\{x: \ x=(i,R_i,x_{l_i+1},x_{l_i+2},...)\in \Sigma_N^{+}\}.$

Let $\widetilde{T}(x)=\left\{\begin{array}{l}(i,R_i,R_i,R_i,...) \ \textnormal{if} \ T(x)\in W_i \  \textnormal{but} \ x\notin W_{i,i}\\
 (i,x_{l_i+1},...) \ \textnormal{if} \ x\in W_{i,i}
\end{array}
\right.$

For all $x\in K$ if $\widetilde T^{n}(x)\in W_{x_1,x_1}$ for all positive $n$, then $x = (x_1, R_{x_1},  R_{x_1}, ...)$ and $x$ is fixed by $\widetilde T$. Otherwise there exists some $n_0\ge 0$ such that $\widetilde T^{n_0+1}(x)\notin W_{x_1,x_1}$, in which case $\widetilde T^{n_0+1}(x)$ is also fixed by $\widetilde T$.
\end{proof}

Theorem \ref{res1} is a direct consequence of the following theorem and of Proposition \ref{eps}.

\begin{thm}\label{endo}
Given $T:K\rightarrow K$ an endomorphism and $K_1,...,K_N$ disjoint Cantor sets with $K\displaystyle=\bigcup_{i=1}^{N}K_i$, there exists an endomorphism $\widetilde T:K\rightarrow K$ such that:
$$\widetilde T(x)\in K_i \Longleftrightarrow T(x)\in K_i,$$
and every orbit $\widetilde T$ is finally periodic.
\end{thm}

\begin{proof}
The proof is given by induction on $N$. For N=1 it suffices to take $\widetilde T(x)$ to be the identity.
Suppose now that the theorem is valid for $N-1$: If for all $i=1,...,N$, there exists $x\in K_i$ such that $T(x)\in K_i$, then by the previous lemma, the theorem is valid.  Assume then that there exists some $j$, with $1\leq j \leq N$ such that, for all $x\in K_j$ we have that $T(x)\notin K_j$. Let $V_l, 1\le l\le N-1$ be disjoint Cantor sets satisfying $\bigcup_{l=1}^{N-1}V_l= K_j$ and that, for each $l$, there exists $i_l$ such that $V_l\subset K_j\cap T^{-1}(K_{i_l})$. Since we allow for $i_{l_1}$ to be equal to $i_{l_2}$ even if $l_1\not=l_2$, by proposition \ref{div} we can further assume that each $V_l$ is nonempty.

Since $T$ is surjective, the sets $T^{-1}(K_j)$ is nonempty.  Let $U_l, 1\le l\le N-1$ be again disjoint nonempty Cantor sets satisfying $\bigcup_{l=1}^{N-1}U_l= T^{-1}(K_j)$ and such that, for each $l$, there exists $k_l$ such that $U_l\subset K_{k_l}\cap T^{-1}(K_j)$.

Let, for each $1\le l\le N-1,\, h_l$ be a homeomorphism between $U_l$ and $V_l$. Define $\widehat{T}:K\rightarrow K$ by 

\hspace{1.8cm} $\widehat{T}(x)=\left\{\begin{array}{l}
h_l(x) \ \  \textnormal{if} \ x\in {U_l} \\
T(x)  \ \ \textnormal{if} \ x\notin T^{-1}(K_j)\\
\end{array}
\right.$

Let $g:K\setminus K_j\rightarrow K\setminus K_j$ be

\hspace{1.8cm} $g(x)=\left\{\begin{array}{l}
T(x) \ \  \textnormal{if} \ x\notin {U_{l}} \\
\widehat T^2(x) \ \textnormal{if} \ x\in {U_{l}} 
\end{array}
\right.$

By the induction hypothesis, there exists an endomorphism $\widetilde g:K\setminus K_j\rightarrow K\setminus K_j$ such that for $i=1,...,j-1,j+1,...,N$:
$$\widetilde g(x)\in K_i\Longleftrightarrow g(x)\in K_i,$$
\noindent and every orbit is finally periodic.

Finally, define $\widetilde T:K\rightarrow K$ by 

\hspace{1.8cm} $\widetilde T(x)=\left\{\begin{array}{l}
\widetilde g(x) \ \ \ \ \ \ \ \ \ \ \ \ \ \textnormal{if} \ x\notin K_j\cup T^{-1}(K_j)\\
\widehat T(x) \ \  \ \ \ \ \ \ \ \ \ \ \textnormal{if} \ x\in{U_l} \\
\widetilde g\left(h_l^{-1}(x)\right)\ \ \ \ \textnormal{if}\ x\in V_l 
\end{array}
\right.$

Since for all $x\notin K_j, \, \widetilde g(x)$ is equal to either $\widetilde T(x)$ or to $\widetilde T^2(x)$, and all $\widetilde g$ orbits are finally periodic, then all $\widetilde T$ orbits are also finally periodic. 

\end{proof}

\section{Homeomorphisms}

Theorem \ref{res2} is again a direct consequence of Proposition \ref{eps} and the following result:

\begin{thm}\label{hom}
Given $T:K\rightarrow K$ be a homeomorphism and let $K_1,...,K_N$ be disjoint Cantor sets with $K\displaystyle=\bigcup_{i=1}^{N}K_i$, there exists a homeomorphism $\widetilde T:K\rightarrow K$ such that:
$$\widetilde T(x)\in K_i \Longleftrightarrow T(x)\in K_i,$$
\noindent and the $\omega$-limit of every orbit of $\widetilde T$ is a periodic orbit.
\end{thm}

The proof of this theorem is also by induction and, similarly to Theorem \ref{endo}, we analyzed the next particular case (Lemma 4.1) in order to use it in the first case its proof. 

\begin{lem} \label{lem}
Let $T\in Hom(K)$ and $K_1, K_2,\ldots, K_N$ be disjoint Cantor sets with $K\displaystyle=\bigcup_{i=1}^{N}K_i$. Consider that for every $1\leq i\leq N$, at least one of the following properties are satisfied:

\begin{enumerate}[\upshape (i)]
\item For any $y\in K_i$ we have $T^{-1}(y)\in K_i$,
\item For any $x\in K_i$ we have $T(x)\in K_i$.
\end{enumerate}

Then there exists a homeomorphism $\widetilde T:K\rightarrow K$ such that
$$\widetilde T(x)\in K_i \Longleftrightarrow T(x)\in K_i,$$
and every orbit of $\widetilde T$ converges to a fixed point. 
\end{lem} 

\begin{proof}
Let $K_{R_l}$ be the Cantor sets that do not satisfy the property i), with $l=1,...,r$ for $1\leq r\leq N$, $K_{A_q}$ be the Cantor sets that do not satisfy property ii), with $q=1,...,n$,  for $1\leq n\leq N$, and let $\overline K= K\setminus \left(\left(\bigcup_{l=1}^{r}K_{R_l}\right)\cup\left(\bigcup_{q=1}^{n}K_{A_q}\right)\right)$. Note that, $N\ge q+r$, and if $K_i\subset \overline K$, then $T(K_i)=K_i$. Furthermore, since $\overline K$ is invariant, $\bigcup_{l=1}^{r} K_{R_l}$ properly contains its image and $\bigcup_{q=1}^n K_{A_q}$ is properly contained in its image, $r$ is null if and only if so is $n$.

1) For each $l, q$, define 
$$W_{A_q}^{R_l}=\left\{x\in K_{R_l}: \ T(x)\in K_{A_q}\right\}\neq\emptyset,$$
$$V_{R_l}=\left\{x\in K_{R_l}: \ T(x)\in K_{R_l}\right\}.$$

Note that $(\bigcup_{q=1}^{n}W_{A_q}^{R_l})\cup V_{R_l} = K_{R_l}$, and that $T(V_{R_l})=K_{R_l}$.  By \ref{seq}, for each $l\le r$, there exists a sequence of Cantor sets $\widetilde K_m^l\subset V_{R_l}$ and a point $p_l\in V_{R_l}$ such that, $\displaystyle V_{R_l}=\bigcup_{m=1}^{\infty} \widetilde K^l_m\bigcup \{p_l\}$ with $diam(\widetilde K^l_m)\rightarrow 0$.
	
Consider, for each $m\geq 2$, the homeomorphism $h^l_m:\widetilde K^l_m\rightarrow \widetilde K^l_{m-1}$ and the homeomorphism $\overline{h}^l:\widetilde K^l_1\rightarrow \bigcup_{q=1}^{n}W_{A_q}^{R_l}$.

Define $\widehat{T_{R_l}}:V_{R_l}\rightarrow K_{R_l}$ by 

\hspace{1.8cm} $\widehat{T_{R_l}}(x)=\left\{\begin{array}{l}
h^l_m(x) \ \ \ \ \ \textnormal{if} \ x\in \widetilde K^l_m, \  m\geq2 \\
\overline{h}^l(x) \ \ \ \ \ \ \textnormal{if} \ x\in \widetilde K^l_1\\
p_l \ \ \ \ \ \ \ \ \ \ \textnormal{if} \ x=p_l
\end{array}
\right.$

Note that, $\widehat{T_{R_l}}^{m+1}(x)\notin V_{R_l}$, for $x\in \widetilde K_m^l$ with $x\neq p_l $.

2) For each $K_{A_q}$, there exists a sequence of Cantor sets $\widehat K^q_m\subset K_{A_q}$ and a point $p_q\in K_{A_q}$ such that, $\displaystyle K_{A_q}=\bigcup_{m=1}^{\infty} \widehat K^q_m\bigcup \{p_q\}$ with $diam(\widehat K^q_m)\rightarrow 0$.

Let, for each $m\geq2$, the homeomorphism $h^q_m:\widehat K^q_{m-1}\rightarrow \widehat K^q_m$.

Define $\widehat{T_{A_q}}:K_{A_q}\rightarrow K_{A_q}$ by 

\hspace{1.8cm} $\widehat{T_{A_q}}(x)=\left\{\begin{array}{l}
h^q_m(x) \ \ \ \ \ \textnormal{if} \ x\in \widehat K_{m-1}^q, \ m\geq2 \\
p_q \ \ \ \ \ \ \ \ \ \ \textnormal{if} \ x=p_q\\
\end{array}
\right.$ 

Note that, for all $x\in K_{A_q}, \widehat T_{A_q}(x)$ converges to fixed point $p_q\in K_{A_q}$.
	
To define the homeomorphism in K we used, for each ${A_q}$ with $q=1,\ldots,n$ and $1\leq n\leq N$, there exists a homeomorphism $h_{A_q}:W_{A_q}^{R_1}\cup W_{A_q}^{R_2}\cup ...\cup W_{A_q}^{R_l}\rightarrow \widetilde K_1^l\subset K_{A_q}$.

Finally, using 1) and 2), we can define $\widetilde{T}:K\rightarrow K$ by 

\hspace{1.8cm} $\widetilde{T}(x)=\left\{\begin{array}{l}
\widehat{T_{R_l}}(x)  \ \ \ \ \ \ \ \ \textnormal{if} \ x\in V_{R_l}, \ l=1,...,r\\
\widehat{T_{A_q}}(x)\ \ \ \ \ \ \ \ \textnormal{if} \ x\in K_{A_q}, \ q=1,...,n\\
h_{A_q}(x) \ \ \ \ \ \ \ \ \textnormal{if} \ x\in K_{R_l}\setminus V_{R_l}, \ l=1,...,r\\
\widehat{T}(x)= x  \ \ \ \ \ \textnormal{if} \ x\in \overline{K}
\end{array}
\right.$
\end{proof}

\begin{proof}{\it of Theorem \ref{hom}}
 
The proof is by induction on $N$. The case where $N=1$ is trivial, as it suffices to take $\widetilde T$ to be the identity. Assume the result is true for $N-1$. There are two possibilities: First, for each $i, \, K_i$ is either forward or backward invariant. Or there exists some $j$ such that $T^{-1}(K_j)\not= K_j$ and $T(K_j)\not=K_j$. In the first possibility, the result follows from Lemma \ref{lem}. So consider now the second possibility.

Let $U_0= K_j\cup T^{-1}(K_j)$ and let $U_l, 1\le l\le N-1$ be again disjoint nonempty Cantor sets satisfying $\bigcup_{l=0}^{N-1}U_l= T^{-1}(K_j)$ and such that, for each $l>0$, there exists $k_l\not=j$ such that $U_l\subset K_{k_l}\cap T^{-1}(K_j)$. Let $V_0=U_0$, and let $V_l, 1\le l\le N-1$ be disjoint Cantor sets satisfying $\bigcup_{l=0}^{N-1}V_l= K_j$ and that, for each $l>0$, there exists $i_l\not=0$ such that $V_l\subset K_j\cap T^{-1}(K_{i_l})$. Again,  by proposition \ref{div} we can further assume that each $U_l, V_l, l>0$ is nonempty and $V_0$ is empty only if so $U_0$. 

Let, for each $0\le l\le N-1,\, h_l$ be a homeomorphism between $U_l$ and $V_l$. Define $\widehat{T}:K\rightarrow K$ by 

\hspace{1.8cm} $\widehat{T}(x)=\left\{\begin{array}{l}
h_l(x) \ \  \textnormal{if} \ x\in {U_l}, \\
T(x)  \ \ \textnormal{if} \ x\notin T^{-1}(K_j),\\
\end{array}
\right.$

The rest of the proof is identical to the proof of Theorem \ref{endo}.
\end{proof}

\begin{proof} {\it of Corollary \ref{co:intro}}
It is a well known fact in ergodic maximization theory that given an endomorphisms $T$ of a compact space and a continuous observable $f$, there is always at least one ergodic $T$ invariant measure which is $f$ maximizing. But if $\widetilde T$ is a transformation such that for every $x$ the orbit by $\widetilde T$ is either finally periodic, or if the $\omega$-limit of $x$ is a periodic orbit, then it follows that any $\widetilde T$ invariant ergodic measures is supported on a periodic orbit.
\end{proof}

\subsection*{Acknowledgements}
The first author was supported by Capes, the second author was supported by CNPq and the third author was partially supported by CNPq and FAPESP

\end{document}